\documentclass[12pt,a4paper]{article}
\usepackage[T1]{fontenc}
\usepackage[margin=3cm,footskip=1cm]{geometry}

\usepackage{amsmath,amssymb,amsthm,bm,mathtools,xspace,booktabs,url}
\usepackage{graphicx,etoolbox}
\usepackage{dsfont}
\usepackage{hyperref}
\usepackage{bbold}
\usepackage{tikz}
\usepackage[shortlabels]{enumitem}
\usepackage{xcolor}
\usepackage{mathrsfs} 
\usepackage{color}

\newcommand{\dd}{\mathrm{d}}

\makeatletter

\usepackage{amssymb}

\usepackage{cleveref}
\usepackage{soul}

\newcommand{\ud}{\, \mathrm{d}}  %  "d" for integrals
\usepackage{caption}
\usepackage[raggedright]{titlesec}

\numberwithin{equation}{section}
\usepackage{amsfonts}
\theoremstyle{plain} 
\newtheorem{theorem}{Theorem}[section]
\newtheorem{Lemma}[theorem]{Lemma}
\newtheorem{Proposition}[theorem]{Proposition}

\theoremstyle{definition} 

\newtheorem{Remark}[theorem]{Remark}
 {
      \theoremstyle{plain}
      \newtheorem{Assumption}{Assumption}
  }

\usepackage{authblk}

\makeatletter
\newcommand\CorrespondingAuthor[1]{
  \begingroup
  \def\@makefnmark{}
  \footnotetext{Corresponding author: #1}
  \endgroup
}
\makeatother

\makeatletter
\renewenvironment{abstract}{%
  \small%
  \begin{center}%
    \bfseries \abstractname\vspace{-.5em}\vspace{\z@}%
  \end{center}%
  \quote%
}
{
\endquote}
\makeatother

\usepackage[above,below,section]{placeins}
\usepackage[all]{onlyamsmath}

\usepackage{multirow}
\usepackage{xcolor}
\usepackage{color}

\usepackage{mathrsfs}
\usepackage{pdfpages}

\definecolor{darkmagenta}{rgb}{0.5,0,0.5}
\definecolor{darkgreen}{rgb}{0,0.4,0.2}
\definecolor{darkblue}{rgb}{0,0,0.6}
\definecolor{darkred}{rgb}{0.8,0,0}
\definecolor{mellow}{rgb}{.847, 0.72, 0.525}

\begin{document}

\title{{Wong--Zakai approximations with convergence rate for stochastic differential equations \\with regime switching}}
\author{Giang T. Nguyen, Oscar Peralta}

\maketitle
\begin{abstract}
{We construct Wong--Zakai approximations of time--inhomogeneous stochastic differential equations with regime switching (RSSDEs), and provide a convergence rate.
%Given a family of finite-variation processes $\{\mathcal{F}^{\lambda}\}_{\lambda\ge 0}$ that converge strongly to a standard Brownian motion $\mathcal{B}$, we construct pathwise approximations for regime-switching, time-inhomogeneous stochastic differential equations in the Wong-Zakai sense. Moreover, we determine the rate of strong convergence to the solutions of such regime-switching SDEs, showing that this rate is almost as good as that of $\{\mathcal{F}^{\lambda}\}_{\lambda\ge 0}$ to $\mathcal{B}$. 
In the proposed approximations, the standard Brownian motion driving the time-inhomogeneous RSSDEs is replaced by a family of finite--variation processes $\{\mathcal{F}^{\lambda}\}_{\lambda > 0}$. We show that if $\mathcal{F}^{\lambda}$ strongly converges to $\mathcal{B}$ at rate $\delta(\lambda)$, then the Wong--Zakai approximation strongly converges to the original solution of the time--inhomogeneous RSSDE at rate $\delta(\lambda) \lambda^{\varepsilon}$, for any $\varepsilon > 0$. This is the first paper on Wong--Zakai approximations for time--inhomogeneous RSSDEs, and significantly extends the counterparts for time--homogeneous SDEs without regime switching in R\"{o}misch and Wakolbinger~\cite{romisch1985lipschitz}.}  

%Here, two key techniques are the Lamperti transform for stochastic differential equations with regime switching, and a pathwise analysis of their solutions.   
% 
\end{abstract}

\section{Introduction}
%
%\giang{I've swapped the order of [the process definition] and [its applications], to emphasize the mathematical model.}  
%The definition of a regime-switching stochastic differential equation (RSSDE) is as follows.
Let $\mathcal{B}=\{B_t\}_{t\ge 0}$ be a standard Brownian motion and let $\mathcal{J}=\{J_t\}_{t\ge 0}$ be a c\`{a}dl\`{a}g jump process with finite state space $\mathcal{E}$. Suppose that $\mathcal{J}$ is independent of $\mathcal{B}$ and has finite jump activity on compact intervals. For~$i\in\mathcal{E}$, let $\mu_i:\mathds{R}_+\times\mathds{R}\mapsto\mathds{R}$ and $\sigma_i:\mathds{R}_+\times\mathds{R}\mapsto\mathds{R}$, where $\mathds{R}_{+}  = [0,\infty)$. A \emph{regime--switching time-inhomogeneous stochastic differential equation} (RSSDE) takes the form
\begin{align} 
  \label{eq:RSSDEX1}
\dd X_t &= \mu_{J_t} (t,X_t) \dd t + \sigma_{J_t} (t,X_t) \dd B_{t}, \quad {X_0 = x_0 \in \mathds{R}}.
\end{align}
RSSDEs are used to model systems in which randomness arise from two sources: continuous (diffusive) and discrete (environmental). These processes are widely used in areas such as economics, finance, ecology, and automatic control \cite{yin2009hybrid}. Although (\ref{eq:RSSDEX1}) is well--defined in the context of It\^{o} stochastic differential equations (SDE), most of its properties such as existence and uniqueness of solutions and stability do not follow from standard SDE theory and thus need to be properly verified; see \cite{mao1999stability} for {the particular case} when $\mathcal{J}$ is Markovian.

%\giang{I've also swapped the order of [contributions] and [definitions of strong convergence \& strong convergence rate], to promote the contributions.} 

Suppose that there exists a unique continuous solution $\mathcal{X}=\{X_t\}_{t\ge 0}$ of (\ref{eq:RSSDEX1}). {Our goals in paper are twofold.} First, we construct a family of {the so--called Wong--Zakai approximations} $\{\mathcal{X}^\lambda\}_{\lambda > 0} = \{\{X^{\lambda}_t\}_{t \geq 0}\}_{\lambda > 0}$. {On a fixed time interval $[0,T]$, $T < \infty$, we replace the standard Brownian motion $\mathcal{B}$ in \eqref{eq:RSSDEX1} by} a family of finite-variation processes $\{\mathcal{F}^\lambda\}_{\lambda > 0} = \{\{F^{\lambda}_t\}_{t \geq 0}\}_{\lambda > 0}$ that strongly converges to $\mathcal{B}$; that is, 
\[
	\lim_{\lambda\rightarrow\infty}\sup_{t\in [0,T]}\left| B_t - F^\lambda_t 
	\right|=0 \quad\mbox{almost surely.}
\]
{For each $\lambda > 0$, consider the Wong--Zakai approximation} $\mathcal{X}^{\lambda}$ that is the solution of the following regime-switching SDE 
\begin{align} 
	\label{eqn:xlambda}
	 \dd X_t^\lambda &= \left(\mu_{J_t} (t,X_t^\lambda) - \frac{1}{2}\sigma_{J_t}(t,X^\lambda_t)\partial_2\left[\sigma_{J_t}(t,X^\lambda_t)\right] \right) \dd t + \sigma_{J_t} (t,X_t^\lambda) \dd F_{t}^\lambda,
\end{align}  
where $\partial_k$ denotes the partial differential operator with respect to the $k$th entry. We show in Proposition~\ref{theo:Xlambda} that this solution is unique. 

{Second, under certain} {regularity conditions} on $\mu_i$ and $\sigma_i$, {and boundedness on~$\mathcal{J}$}, we show that the Wong--Zakai approximations $\{\mathcal{X}^{\lambda}\}_{\lambda > 0}$ converge to the solution $\mathcal{X}$ of~(\ref{eq:RSSDEX1}). {More precisely,} we obtain the rate of strong convergence of $\mathcal{X}^\lambda$ to $\mathcal{X}$, and show that it is \emph{almost the same} as that of $\mathcal{F}^\lambda$ to $\mathcal{B}$. The precise rate of strong convergence and the conditions under which this holds are stated in Theorem~\ref{th:mainmainresult}, our main result.
 
%\giang{I've stated main results here, and incorporated in it the definition of convergence rate.}

\begin{theorem}
	\label{th:mainmainresult}
Under Assumptions \ref{ass:existenceSDEs}--\ref{ass:tailN}, suppose that there exists a family of finite--variation processes $\{\mathcal{F}^{\lambda}\}_{\lambda > 0}$ such that for some $\delta:\mathds{R}_+\mapsto\mathds{R}_+$ with $\lim_{\lambda\rightarrow \infty}\delta(\lambda)=0$, we have that for all $q>0$ 
\begin{align*} 
	\mathds{P}\left(\sup_{t\in [0,T]}|F^\lambda_t - B_t| \ge \alpha \delta(\lambda)\right)=o(\lambda^{-q}),
\end{align*}
where $\alpha=\alpha(q, T)>0$ is some constant dependent on $q$ and $T$ only. For $\lambda > 0$, let $\mathcal{X}^{\lambda}$ be the unique solution $\mathcal{X}^\lambda$ to the regime-switching SDE 
\begin{align}
\dd X_t^\lambda & = \left(\mu_{J_t} (t,X_t^\lambda) - \frac{1}{2}\sigma_{J_t}(t,X^\lambda_t)\partial_2 \left[\sigma_{J_t}(t,X^\lambda_t)\right]\right) \dd t + \sigma_{J_t} \left(t,X_t^\lambda\right) \dd F_{t}^\lambda, \\
X^\lambda_0 & = x_0.
\end{align}
Then,
\begin{itemize} 
	\item[(i)] the family of processes $\{\mathcal{X}^\lambda\}_{\lambda > 0}$ converges strongly to the unique solution $\mathcal{X}$ of the regime--switching time--inhomogeneous SDE~\eqref{eq:RSSDEX1}
	\item[(ii)] moreover, for all $q, T >0$ there exists $\gamma=\gamma(q,T)>0$ such that for all $\varepsilon > 0$ 
\begin{align} 
\label{eq:rateRSSDEmain1}
  \mathds{P}\left(\sup_{t\in [0,T]}|X^\lambda_t - X_t| \ge \gamma \delta(\lambda)\lambda^{\varepsilon}\right) = o(\lambda^{-q}).
\end{align} 
\end{itemize}
\end{theorem}
\noindent {Thus, Theorem~\ref{th:mainmainresult} states that if $\mathcal{F}^{\lambda}$ strongly converges to $\mathcal{B}$ at rate $\delta(\lambda)$, then the Wong--Zakai approximation $\mathcal{X}^{\lambda}$ strongly converges to $\mathcal{X}$ at rate $\delta(\lambda) \lambda^{\varepsilon}$, for any $\varepsilon > 0$.} 

{Details of Assumptions~\ref{ass:existenceSDEs}--\ref{ass:tailN} can be found in later sections. Briefly speaking, Assumption~\ref{ass:existenceSDEs} is about linear growth and local Lipschitz continuity, and is standard for ensuring existence~{and uniqueness} of solutions of SDEs.  Assumption~\ref{ass:wongzakai} provides bounds for the diffusion coefficient $\sigma$ and its derivative, and is standard for Wong--Zakai approximations. Assumption~\ref{ass:mustarunifLip} is needed specifically because of the time--inhomogeneity feature,
% While Assumption \ref{ass:mustarunifLip} is a slightly \emph{stronger} restriction w.r.t. to the original properties of $\mu_i^*$ in Lemma \ref{th:mustarLipschitz1}, 
and does not automatically follow from the properties of $\mu$ and $\sigma$ described in Assumptions~\ref{ass:existenceSDEs} and \ref{ass:wongzakai}. 
% (see, e.g., the case when $\mu(t,x) = x$ and $\sigma(t,x) = (2 + \cos(x))^{-1}$). 
Assumption~\ref{ass:tailN} arises due to the regime switching feature, and bounds the tail of the jump distribution of $\mathcal{J}$ over a compact time interval. This assumption allows us to obtain the convergence rate in the regime--switching case, without being overly restrictive: it holds whenever $\mathcal{J}$ has a bounded jump intensity. Cases where Assumption~\ref{ass:tailN} is satisfied include, but are not limited to, time--homogeneous Markovian, time--inhomogeneous Markovian, and semi--Markovian processes.}

%\giang{I've brought forward the below part, before talking about Wong--Zakai papers.} 

Most related to this paper is the work of R{\"o}misch and Wakolbinger \cite{romisch1985lipschitz}, which determined the rate of strong convergence of the Wong--Zakai approximation for time-homogeneous SDEs {without regime switching}, {under Assumptions~\ref{ass:existenceSDEs} and \ref{ass:wongzakai} reduced to the time--homogeneous setting.} While our strong approximation is {an} extension of 
\cite{romisch1985lipschitz}, our methods significantly differ from theirs in order to handle challenges that the regime--switching and time--inhomongeneity generalizations bring to the problem. 
% {and, as a by product, greatly} simplify the exposition. 
A key technique in our analysis is applying the Lamperti transform~\cite{l64} to regime--switching SDEs \eqref{eq:RSSDEX1}. This was previously applied in Pettersson~\cite{p99} and Zhang~\cite{z94} in the simpler case of time--homogeneous SDEs without regime switching.
The Lamperti transformation leads to clear expressions in the approximation and an insight on how to bound approximation errors at jump epochs of $\mathcal{J}$. 

\subsection{Related literature}

Wong and Zakai~\cite{wz65b, wz65a} were the first to consider almost-sure approximations to solutions of SDEs. Let $\pi$ be a finite partition of $[a,b]$, where $\pi = \{t_1, \ldots, t_{k + 1}\}$, $a = t_1 < t_2 < \ldots < t_{k + 1} = b$, and let $\mathcal{B}^n = \{B^n_t\}_{t \geq 0}$ be a process whose the sample paths coincide with those of $\mathcal{B}$ at~$t_j$ and are linear between them. Then, as the mesh of $\pi$, i.e. $\max_j |t_{j + 1} - t_j|$, tends to $0$, the solutions $\mathcal{Y}^n = \{Y^n_t\}_{t \geq 0}$ of the equation 
\begin{align}
	\label{eqn:xn}
	Y^n_t = y_0 + \int_a^t \mu(s, Y^n_s)\ud s + \int_a^t \sigma(s, Y^n_s)\ud B^n_s, \quad t \in [a,b], 
\end{align}
converge in quadratic mean to $\mathcal{Y} = \{Y_t\}_{t \geq 0}$~\cite{wz65a}, which is a solution of 
\begin{align}
	\label{eqn:x}
	Y_t = y_0 + \int_a^t \mu(s, Y_s) \ud s + \int_a^t \sigma(s, Y_s) \ud B_s + \frac{1}{2} \int_a^t \sigma(s, Y_s) \partial_2[\sigma (s, Y_s)]\ud s.
\end{align} 
The last term on the RHS of~\eqref{eqn:x} is known as the \emph{Wong-Zakai correction term}, and the solution to \eqref{eqn:x} is also known as \emph{the Stratonovich solution to the SDE}.  In~\cite{wz65b}, Wong and Zakai extended their previous result to show that, given any piece-wise smooth approximation $\mathcal{B}^n$ for $\mathcal{B}$, the convergence from $\mathcal{Y}^n$ to $\mathcal{Y}$ holds almost surely. 
 
Subsequently, approximations following this framework are generally referred to as of the \emph{Wong-Zakai type}; see the comprehensive survey by Twardowska \cite{twardowska1996wong} and references therein. We describe briefly three major directions of generalization of Wong--Zakai results. First, the Wong--Zakai framework has been extended to SDEs driven by higher--dimensional Brownian motion processes. Stroock and Varadhan~\cite{sv72} proposed piecewise--linear approximations and showed convergence in probability; Ikeda \emph{et al.~}~\cite{iny77} and Ikeda and Watanabe~\cite{iw92} proved, respectively, convergence in meansquare and uniform convergence over finite time interval. Second, Wong--Zakai approximations have also been obtained for reflected SDEs. Pettersson~\cite{p99} analysed SDEs reflected on the boundary of a convex set, replacing the Brownian motion process with its polygonal approximation, and obtained uniform convergence. Evans and Stroock~\cite{se11}~approximated SDEs reflected in general domains and obtained convergence in law. Zhang~\cite{z14} showed strong convergence of reflected SDE in multi--dimensional general domains. Third, {Wong--Zakai} approximations have been generalized to stochastic partial differential equations, such as the Wong-Zakai theorem obtained by Hairer and Pardoux~\cite{hp15} for one--dimensional parabolic nonlinear stochastic PDEs driven by space--time white noise.  

Results on convergence rates in Wong-Zakai approximations are much rarer; we briefly highlight a few. Brzeniak and Flandoli \cite{bf95} proved almost sure convergence for parabolic and hyperbolic evolution equations and provided estimates for the rate of convergence. Later, Brzeniak and Carroll~\cite{bc03} obtained convergence rate for approximations of SDEs in M--type $2$ Banach spaces. More recently, Nakayama and Tappe~\cite{nt18} obtained convergence rate for semilinear stochastic partial differential equation, and Ammou and Lanconelli~\cite{al19} for Stratonovich and It\^{o} quasi--linear SDEs. 
%and \magenta{Viitasaari and Zeng (2020+)} for SDEs driven by fractional Brownian motion. 

\subsection{Paper structure}

We outline in Section~\ref{sec:stochasticcalculus} some results of stochastic calculus for regime-switching processes which will be useful later. In~\ref{sec:potato}, we introduce the Lamperti transformation, $\mathcal{L}$, of a solution of a regime-switching SDE, $\mathcal{X}$, and prove its properties. In~\ref{sec:pathL}, we construct the transformation pathwise and prove its uniqueness. In Section~\ref{sec:rateL}, we apply the Lamperti transformation to the approximating process $\mathcal{X}^{\lambda}$, given in~\eqref{eqn:xlambda}, to obtain $\mathcal{L}^{\lambda}$ and compute its convergence rate to $\mathcal{L}$. We show that, in fact, $\mathcal{L}^{\lambda}$ is the unique c\`{a}dl\`{a}g solution to a regime-switching SDE driven by $\mathcal{F}^\lambda$ with \emph{jumps}. Lipschitz continuity helps us obtain the almost sure convergence, and subsequently the convergence rate, from $\mathcal{X}^{\lambda}$ to $\mathcal{X}$, giving rise to Theorem~\ref{th:mainmainresult}. Finally, in Section~\ref{sec:applications} we state a direct implication of Theorem \ref{th:mainmainresult}, with respect to recent approximations of Markov-modulated Brownian motion~\cite{latouche2015morphing, nguyen2019strong}. 

\section{Preliminaries and notation}	\label{sec:stochasticcalculus}
Here, we briefly summarize {key results in} stochastic calculus for regime-switching processes {needed for this paper}. We write $\partial_{k\ell}\equiv \partial_k\partial_\ell$ to denote the partial differential operator first w.r.t. to the $k$th entry and then w.r.t. the $\ell$th entry. Furthermore, $\mathcal{C}^{i,j}$ denotes the class of continuous $\mathds{R}^2 \mapsto \mathds{R}$ functions with continuous $i$th derivative with respect to the first parameter and continuous $j$th derivative with respect to the second parameter. 

Let $\mathcal{Z}=\{Z_t\}_{t\ge 0}$ be an almost surely continuous stochastic process; we write $\{[Z]_t\}_{t\ge 0}$ to denote the quadratic variation process of $\mathcal{Z}$. 

\begin{theorem}[It\^o's formula] 
	Let $f\in \mathcal{C}^{1,2}$. Then, for all $t\ge  0$
\begin{align*}
f(t,Z_t) & = f(0,Z_0) + \int_0^t \partial_1[f (s, Z_s)] \dd s + \int_0^t \partial_2[f (s, Z_s)] \dd Z_s + \frac{1}{2}\int_0^t \partial_{22}[f (s, Z_s)] \dd [Z]_s\quad\mbox{a.s.}
\end{align*}
\end{theorem}

\begin{Remark}
\emph{If $\mathcal{Z}$ is has finite variation, then $[Z]\equiv 0$. If $\mathcal{Z}=\mathcal{B}$, then $[Z]\equiv \mbox{Id}$.} 
\end{Remark}
\begin{theorem}[It\^o's formula for regime-switching processes] 
	\label{th:itoRS1}
	For $i\in\mathcal{E}$, let $f_i\in \mathcal{C}^{1,2}$. Then, for all $t\ge  0$
\begin{align}
f_{J_t}(t,Z_t) & = f_{J_0}(0,Z_0) + \int_0^t \partial_1[f_{J_s}(s, Z_s)] \dd s + \int_0^t \partial_2 [f_{J_s} (s, Z_s)] \dd Z_s + \frac{1}{2}\int_0^t  \partial_{22}[f_{J_s}(s, Z_s)]\dd [Z]_s\nonumber\\
&\quad + \sum_{s\le t: J_s\neq J_{s-}} \left(f_{J_s}(s,Z_s) - f_{J_{s-}}(s,Z_s)\right)\quad\mbox{a.s.} \label{eq:RSito1}
\end{align}
\end{theorem}
\begin{proof}
This follows by concatenation. Let $t_1<t_2<\dots <t_n$ denote the jump epochs of $\mathcal{J}$ on $(0,t]$, and let $t_0:=0$,  % (within this proof) 
$t_{n+1}:=t$. Let $j_i=J_{t_i}$, $0\le i\le n$. Notice that
\begin{align}
& f_{J_t}(t,Z_t) - f_{J_0}(0,Z_0) = f_{J_n}(t_{n+1},Z_{t_{n+1}})- f_{j_0}(t_0,Z_{t_0})\nonumber\\
& = \sum_{k=0}^n \left( f_{j_k}(t_{k+1},Z_{t_{k+1}}) - f_{j_k}(t_{k},Z_{t_{k}})\right) + \sum_{k=1}^n \left( f_{j_k}(t_{k},Z_{t_{k}}) - f_{j_{k-1}}(t_{k},Z_{t_{k}})\right)\nonumber\\
& =\sum_{k=0}^n  \left(f_{j_k}(t_{k+1},Z_{t_{k+1}}) - f_{j_k}(t_{k},Z_{t_{k}})\right) + \sum_{s\le t: J_s\neq J_{s-}} \left(f_{J_s}(s,Z_s) - f_{J_{s-}}(s,Z_s)\right).\label{eq:RSitoaux1}
\end{align}
Then, (\ref{eq:RSito1}) follows by applying It\^o's formula to each term in the first sum of (\ref{eq:RSitoaux1}).
\end{proof}

Note that the concatenation method used in the proof of Theorem \ref{th:itoRS1} holds for both Markovian and non-Markovian jump processes. This technique can be found in the proof of Lemma 3 in \cite[Section II.2.1]{skorokhod1989asymp}, where an alternative form of It\^o's formula is provided for the particular case of regime-switching diffusions with an underlying Markovian jump mechanism; see \cite[Equation (2.8)]{yin2009hybrid} for a related formula.

\begin{Remark}
\emph{Since $\mathcal{F}^\lambda$ converges strongly to $\mathcal{B}$, for each $t\ge 0$ we have 
\begin{align}
	\label{eq:faux01}
	\lim_{\lambda\rightarrow\infty}f_{J_t}(t,F^\lambda_t) - f_{J_t}(t,B_t)=0\quad \mbox{a.s.}
\end{align}
Applying Theorem \ref{th:itoRS1} to both terms in (\ref{eq:faux01}) 
%together with the Dominated Convergence Theorem\footnote{I was slightly imprecise here: to use the DCT you need to bound $\sup_{s\le t}|F_s^\lambda|$ in terms of $\sup_{s\le t}|B_s|$. One can appeal to our previous rate of convergence paper and Borel-Cantelli to prove this.} 
we obtain
\[\lim_{\lambda\rightarrow \infty} \int_0^t \partial_2[f_{J_s}(s, F^\lambda_s)] \dd F^\lambda_s = \int_0^t \partial_2[f_{J_s} (s, B_s)]\dd B_s + \frac{1}{2}\int_0^t \partial_{22}[f_{J_s} (s, B_s)]\dd s\quad \mbox{a.s.}\]
In particular, if $\psi_i \in \mathcal{C}^{1,1}$, then by taking $f_i(t,x) := \int_0^x \psi_i(t,y)\dd y$
 we have
\begin{align}
	\label{eq:RSSI01}
\lim_{\lambda\rightarrow \infty} \int_0^t \psi_{J_s}(s, F^\lambda_s)\dd F^\lambda_s = \int_0^t \psi_{J_s}(s, B_s)\dd B_s + \frac{1}{2}\int_0^t \partial_2[\psi_{J_s} (s, B_s)] \dd s\quad \mbox{a.s.}
\end{align}
As seen in (\ref{eq:RSSI01}), whenever we compute the stochastic integral w.r.t. a finite-variation process with regime switching, the limit as $\lambda\rightarrow\infty$ is equal to the same stochastic integral w.r.t. the Brownian motion \emph{plus} a correction term.} 
\end{Remark}
Now, consider the regime-switching SDE given by
\begin{align}
	\label{eq:RSSDEX4}
\dd Y_t &= \mu_{J_t} (t,Y_t) \dd t + \sigma_{J_t} (t,Y_t) \dd Z_t,\quad Y_0=x_0.
\end{align}

In order to guarantee the existence of a unique c\`{a}dl\`{a}g solution of (\ref{eq:RSSDEX4}), we need the following concepts. Let $g:\mathds{R}_+\times\mathds{R} \mapsto \mathds{R}$. We say that the function $g$ is:
\begin{itemize}
\item \emph{linearly growing} if there exists $E >0$ such that
\[g(t, x)\le E(1+t + |x|)\quad\forall t\ge 0, x\in\mathds{R}; \]
\item \emph{locally Lipschitz continuous} if for each compact set $\mathcal{K} \subset\mathds{R}_+\times\mathds{R}$, there exists $M_{\mathcal{K}}>0$ such that
  \begin{align*}
  \Big|g(t_1,x_1)-g(t_2,x_2)\Big| &\le M_{\mathcal{K}} \left(|t_1 - t_2|+|x_1-x_2|\right)\quad \forall (t_1, x_1),(t_2,x_2) \in \mathcal{K};
  \end{align*}
  \item \emph{Lipschitz continuous on its second entry} if there exists $M>0$ such that
    \begin{align*}
  \Big|g(t,x_1)-g(t,x_2)\Big| &\le M |x_1-x_2|\quad \forall t\ge 0, x\in\mathds{R}.
  \end{align*}
\end{itemize}
Note that a function being Lipschitz continuous on its second entry does not necessarily imply linear growth nor local Lipschitz continuity (e.g. $g(t,x)=\sqrt{t}+t^2$).
\begin{Assumption}
	\label{ass:existenceSDEs}
For all $i\in\mathcal{E}$, $t\ge 0$ and $x\in\mathds{R}$:
\begin{enumerate} 
  \item[\emph{(i)}] $\mu_i$ is in $\mathcal{C}^{0,0}$, linearly growing and locally Lipschitz continuous,
\item[\emph{(ii)}]
 	 $\sigma_i$ is in $\mathcal{C}^{1,1}$, linearly growing, locally Lipschitz continuous, and Lipschitz continuous on its second entry,
  \item[\emph{(iii)}] $\partial_1[\sigma_i]$ and $\partial_2[\sigma_i]$ are locally Lipschitz continuous.
\end{enumerate}
\end{Assumption}
Assumption~\ref{ass:existenceSDEs} corresponds to classic conditions of existence {and uniqueness} of the solution of an SDE {without regime switching}~\cite[Section V.12]{rogers2000diffusions}. {For} regime-switching SDEs, existence and uniqueness of solutions between jump times of $\mathcal{J}$, together with a continuous concatenation of paths, yields the existence and uniqueness of $\mathcal{Y}=\{Y_t\}_{t\ge 0}$. In Section \ref{sec:pathL}, we make explicit this concatenation technique for the Lamperti transform of the solution of the regime--switching SDEs (\ref{eq:RSSDEX1}).
\section{Lamperti transform for regime-switching SDEs}
	\label{sec:Lamperti}
In general, \emph{Lamperti transform} refers a class of transformations of solutions of SDEs that lead to the diffusion term in the resulting processes being independent of the state~\cite{l64}. 

\subsection{Functional construction} 
	\label{sec:potato}
	
Here, we introduce a Lamperti transformation for a regime-switching SDE. Denote by $\{t_k\}_{k\ge 0}$ the jump times of $\mathcal{J}$, with $t_0:=0$,  and we write $J_k$ for $J_{t_k}$ for $k\ge 0$ whenever there is no ambiguity. Let $\mathcal{X}$ be the unique solution to~\eqref{eq:RSSDEX1}. To construct the transformation, we need an extra assumption which is also  in classic papers dealing with Wong-Zakai approximations \cite{wz65b, romisch1985lipschitz}.
\begin{Assumption}
	\label{ass:wongzakai}
For $i \in \mathcal{E}$, 
\begin{itemize} 
	\item[\emph{(i)}] there exist constants $v$ and $V$ such that $0<v<\sigma_i<V<\infty$; 
	\item[\emph{(ii)}] there exists a constant $K>0$ such that $|\partial_1[\sigma_i]|/\sigma_i^2\le K $.
\end{itemize} 
\end{Assumption}

\begin{theorem} 
	[Lamperti transform for regime-switching SDEs]
	\label{th:Lamperti1}
Suppose Assumptions \ref{ass:existenceSDEs} and \ref{ass:wongzakai} hold. For $t\ge 0$ and $i \in \mathcal{E}$, define
\begin{align}
	\label{eq:defht1} 
h_{i}(t, x) := \int_{x_0}^x \frac{1}{\sigma_{i}(t,y)}\dd y,\quad x\in\mathds{R},
\end{align}
and $L_t := h_{J_t}(t, X_t)$. Then, the process $\mathcal{L} = \{L_t\}_{t \geq 0}$ satisfies the regime-switching stochastic differential equation with jumps given by
\begin{align}
	\label{eq:RSSDELamp1}
\dd L_t & = \mu^*_{J_t}(t,L_t)\dd t + \dd B_{t} + \sum_{s\le t: J_s\neq J_{s^-} }\left( h_{J_{s}}\left(s, h^{-1}_{J_{s^-}}(s, L_{s^-})\right) - L_{s^-}\right),\\
 L_0 & = h_{J_0}(0, x_0), \nonumber
\end{align}
where, for each fixed $t$, $h_{i}^{-1}(t, \cdot): \mathds{R} \mapsto \mathds{R}$ denotes the inverse of $h_{i}(t, \cdot): \mathds{R} \mapsto \mathds{R}$ and
\begin{align}
	\label{eq:mustar1}
\mu^*_i(t, \ell) :=  \partial_1\left[h_{i} \left(t, h_{i}^{-1}(t, \ell)\right)\right] + \frac{\mu_i\left(t,h_{i}^{-1}(t, \ell)\right)}{\sigma_i\left(t,h_{i}^{-1}(t, \ell)\right)} - \frac{1}{2}\partial_2\left[\sigma_i\left(t,h_{i}^{-1}(t, \ell)\right)\right].
\end{align}
\end{theorem}

Even though $\mathcal{L}$ has discontinuities at $\{t_k\}_{k > 0}$, the important characteristic is that it has a unit diffusion coefficient; it is identically $1$ and independent of $J_t$, $t$, and $L_t$. 

\begin{proof}
Since $1/V<1/\sigma_{i}(t,y)<1/v$, the function $h_{i}(t, \cdot)$ is non-decreasing and thus bijective with a continuous inverse. The boundedness and $\mathcal{C}^{1,1}$ properties of $\sigma_i$ imply that $h_i(\cdot, \cdot) \in \mathcal{C}^{1,2}$. Applying Theorem \ref{th:itoRS1} gives
\begin{align*}
L_t-L_{0} & = h_{J_t}(t,X_t) - h_{J_0}(0,x_0)\\
& = \int_{0}^t \partial_1\left[ h_{J_s} (s, X_s)\right]\dd s + \int_{0}^t \partial_2\left[ h_{J_s}(s, X_s)\right] \dd X_s + \frac{1}{2}\int_{0}^t \partial_{22} \left[h_{J_s} (s, X_s)\right] \dd [X]_s\\
&\quad + \sum_{s\le t: J_s\neq J_{s-}} \left(h_{J_s}(s,X_s) - h_{J_{s-}}(s,X_s)\right) 
\end{align*}

Since $\dd[X]_t = \sigma_{J_t}^2(t,X_t)\dd [B]_{t} = \sigma_{J_t}^2(t,X_t)\dd t,$ we have 
\begin{align*}
L_t-L_{0} & = \int_{0}^t \left(\partial_1 \left[h_{J_s} (s, X_s)\right] + \partial_2 \left[h_{{J_s}}(s, X_s)\right]\mu_{J_{{s}}}({s},X_{s}) + \frac{1}{2} \partial_{22}\left[ h_{{J_s}}(s, X_s)\right] \sigma_{J_{{s}}}^2({s},X_{{s}})\right)\dd s\\
&\quad + \int_{0}^t \partial_2[h_{{J_s}}(s, X_s)]\sigma_{J_{{s}}}({s},X_{{s}})\dd B_s + \sum_{s\le t: J_s\neq J_{s-}} \left(h_{J_s}(s,X_s) - h_{J_{s-}}(s,X_s)\right). 
\end{align*}

Since $\partial_2\left[ h_{{J_t}} (t, X_t)\right]  = 1/\sigma_{J_t}(t,X_t),$ we have 
$
		\partial_2 \left[h_{{J_s}}(s, X_s)\right] \sigma_{J_{{s}}}({s},X_{{s}})\dd B_s = \dd B_s.
$
	Furthermore, as $\mathcal{X}$ is continuous, it follows that 
\begin{align*}
h_{J_s}(s,X_s) & = h_{J_s}(s,X_{s^-}) = h_{J_{s}}\left(s, h^{-1}_{J_{s^-}}(s, L_{s^-})\right),\quad\mbox{and} \quad \\
h_{J_{s-}}(s,X_s) & = h_{J_{s-}}(s,X_{s^-}) = L_{s^-},
\end{align*}
Finally, (\ref{eq:mustar1}) follows by noticing that
\begin{align*}
% 
%\partial_1 \left[h_i (t, X_t)\right] & = \frac{\dd}{\dd t} h_{t,i} (X_t), \quad 
\partial_{22} \left[h_{{J_t}} (t, X_t)\right] = -\frac{\partial_2 \left[\sigma_{J_t}(t,X_t)\right]}{\sigma_{J_t}^2(t,X_t)} \quad\mbox{and} \quad 
X_t = h_{{J_t}}^{-1}(t, L_t),
\end{align*}
and thus $\mathcal{L}$ solves (\ref{eq:RSSDELamp1}).
\end{proof}
\begin{Remark}\label{rem:Lamperti1} 
\emph{Since $h_{i}(t, \cdot)$ is invertible, the Lamperti transform is injective. Thus, if we find the solution $\mathcal{L}$ to (\ref{eq:RSSDELamp1}) and prove its uniqueness, then we can recover the solution $\mathcal{X}$ to (\ref{eq:RSSDEX1}) by taking $X_t := h_{J_t}^{-1}(t, L_t)$ for all $t\ge 0$.}
\end{Remark}
\begin{Lemma}
	\label{th:mustarLipschitz1}
Under Assumptions \ref{ass:existenceSDEs} and \ref{ass:wongzakai}, for each $i\in\mathcal{E}$ the function $\mu_i^*$, defined as in \eqref{eq:mustar1}, is locally Lipschitz continuous and has linear growth. 
\end{Lemma}

\begin{proof}
Note that $\mu^*_i$ is formed by sums and compositions of $\partial_1[h_i]$, $h_i^{-1}(t, \cdot)$, $\mu_i/\sigma_i$ and $\partial_2[\sigma_{i}]$, all locally Lipschitz continuous and linearly growing: the first three by Lemma~\ref{lem:hhinvLipschitz} and the last one by Assumption~\ref{ass:existenceSDEs}(ii, iii); the function $\partial_2[\sigma_{i}]$ is in fact bounded by the Lipschitz continuity on the second entry of $\sigma_i$. This implies that $\mu^*_i$, too, is locally Lipschitz and linearly growing.
\end{proof}

\subsection{Pathwise construction and uniqueness}
	\label{sec:pathL}

By a concatenation argument, Lemma~\ref{th:mustarLipschitz1} implies that under Assumptions \ref{ass:existenceSDEs} and~\ref{ass:wongzakai} a unique c\`{a}dl\`{a}g solution $\mathcal{L}$ to the regime-switching SDE (\ref{eq:RSSDELamp1}) exists. Given its importance, we give a proof of this in this section.

Here, we explicitly consider $\mathcal{L}$ as a path functional of $\mathcal{B}$ and $\mathcal{J}$.
Since $\mathcal{L}$ has a unit diffusion, the pathwise construction of $\mathcal{L}$ between jumps is considerably easier than in \cite{doss1977liens,follmer1981calcul,romisch1985lipschitz}, as shown next. 
For fixed $i\in\mathcal{E}, b\in\mathds{R}, r\ge 0$ consider the process $\mathcal{Y}_{i,b,r} = \{Y_{i, b ,r}(t)\}_{t \geq 0}$ where 
\begin{align}
Y_{i,b,r}(t) & := b + \int_0^t \mu_i^*\left(r+u, Y_{i,b,r}(u) + B_{r+u} - B_r\right) \dd u, \label{eq:zetadef1}
\end{align}
where $\mu^*_i$ is given in~\eqref{eq:mustar1}. Here, the parameter $b$ acts as the starting point, and $r$ a displacement in time with respect to $\mathcal{B}$. By Theorem \ref{th:mustarLipschitz1}, for any given realization of $\mathcal{B}$, a unique solution $\mathcal{Y}_{i,b,r}$ to the system exists~\cite[Section V.2]{rogers2000diffusions}. 

Next, define the process $\mathcal{S} =\{S_t\}_{t\ge 0}$ as follows. Let $S_0=x_0$, and recursively over $n\ge 0$ let
\begin{align}
S_t & := Y_{J_n, S_{t_n}, t_n}(t-t_n) + B_t  - B_{t_n}  \quad\mbox{for}\quad t\in(t_n,t_{n+1}),\label{eqn:St} \\
S_{t_n} &:= h_{J_n}\left(t_n, h_{J_{n-1}}^{-1}(t_n, S_{t_n^-})\right), \quad n\ge 1. 
\label{eqn:Stn}
\end{align}

\begin{theorem}
	\label{th:pathofL}
The process $\mathcal{S}$ is the unique solution to the regime-switching stochastic differential equation with jumps given by~\eqref{eq:RSSDELamp1}. 
\end{theorem}
\begin{proof}
Clearly, the process $\mathcal{S}$ is continuous on $[0,\infty)\setminus\{t_k\}_{k\ge 1}$. Furthermore, since $\lim_{s\downarrow 0}Y_{i, b, r}(s)=b$ for $i\in\mathcal{E},b\in\mathds{R}, r\ge 0$, 
we have $\lim_{s\downarrow t_n}S_s = S_{t_n}$ for $n \geq 0$. In other words, $\mathcal{S}$ is c\`{a}dl\`{a}g. 

Next, for $t\in(t_n,t_{n+1})$
\begin{align*}
S_t  - S_{t_n} & = \left(Y_{J_n, S_{t_n},t_n}(t-t_n) + B_t  - B_{t_n}\right) - S_{t_n}\\
& = \int_{0}^{t-t_n} \mu_{J_n}^*\left(u + t_n, Y_{J_n, S_{t_n},t_n}(u) + B_{u + t_n} - B_{t_n} \right)\dd u + B_t -B_{t_n}\\
& = \int_{t_n}^{t} \mu^*_{J_n}\left(s, Y_{J_n, S_{t_n},t_n}(u - t_n) + B_{u} - B_{t_n}\right)\dd u + B_t - B_{t_n}\\
& = \int_{t_n}^{t} \mu_{J_n}^*(u, S_u)\dd u + B_t - B_{t_n}.
\end{align*}
Thus, $\mathcal{S}$ indeed evolves according to (\ref{eq:RSSDELamp1}) in its continuity points. Since for $n\ge 1$
\begin{align*}
S_{t_n} - S_{t_n^-} = h_{J_n}\left(t_n, h_{J_{n-1}}^{-1}(t_n, S_{t_n^-})\right) - S_{t_n^-},
\end{align*}
the discontinuities of $\mathcal{S}$, too, coincide with those given in (\ref{eq:RSSDELamp1}). 

Finally, the uniqueness of $\mathcal{S}$ follows by the uniqueness of $\mathcal{Y}_{i,b,r}$ for $i\in\mathcal{E},b\in\mathds{R},r\ge 0$, and by the fact that each point of discontinuity, say $S_{t_n} - S_{t_n^-}$, is completely determined by $\mathcal{J}$ and $\{S_t: t\in [0,t_n)\}$.
\end{proof}

Note that the Lamperti transformation introduced in~Theorem~\ref{th:Lamperti1} can be applied to solutions of general stochastic differential equations driven by continuous processes that are not standard Brownian motion. In that case, the pathwise construction above can be generalised to build, and show the uniqueness of, those solutions. In fact, this is what we {shall} need in the next section. 

\section{Rate of convergence for transformed processes}
	\label{sec:rateL}
	
Consider a Wong-Zakai approximation to $\mathcal{X}$, denoted by $\mathcal{X}^{\lambda}$, as given in Eqn~\eqref{eqn:xlambda}. 
	
%	\giang{Demoted the below from Theorem to Proposition, due to the brevity of the proof.} 
	
\begin{Proposition} 
	\label{theo:Xlambda}
The process $\mathcal{X}^\lambda$ defined by \eqref{eqn:xlambda} is the unique solution of 
\begin{align}
	\label{eq:RSSDEXlambda3}
\dd X_t^\lambda &= \left(\mu_{J_t} (t,X_t^\lambda) - \frac{1}{2}\sigma_{J_t}(t,X^\lambda_t)\partial_2\left[\sigma_{J_t}(t,X^\lambda_t)\right] \right) \dd t + \sigma_{J_t} (t,X_t^\lambda) \dd F_{t}^\lambda,\\
 X_0^\lambda & = x_0.
\end{align}
\end{Proposition}
\begin{proof}
By Assumption~\ref{ass:existenceSDEs}, $\sigma_i\partial_2[\sigma_i]$ is a product of two locally Lipschitz functions, so it is itself locally Lipschitz. Furthermore, 
 $\sigma_i\partial_2[\sigma_i]$ is the product of a bounded function and one with linear growth, so that $\sigma_i\partial_2[\sigma_i]$ has linear growth. This implies uniqueness of the solution to the RSSDE (\ref{eq:RSSDEXlambda3}).
\end{proof}

To show strong convergence from $\mathcal{X}^{\lambda}$ to $\mathcal{X}$, we first apply the Lamperti transformation given in Theorem~\ref{th:Lamperti1} to $\mathcal{X}^{\lambda}$, $\mathcal{L}^{\lambda} = \{L^{\lambda}_t\}_{t \geq 0} = \{h_{J_t}(t, X^{\lambda}_t)\}_{t \geq 0}$, which satisfies the SDE 
\begin{align}
	\label{eq:approx}
\dd L^{\lambda}_t & = \mu^*_{J_t}(t,L^{\lambda}_t)\dd t + \dd F^{\lambda}_{t} + \sum_{s\le t: J_s\neq J_{s^-} }\left( h_{J_{s}}\left(s, h^{-1}_{J_{s^-}}(s, L^{\lambda}_{s^-})\right) - L^{\lambda}_{s^-}\right),\\
 L^{\lambda}_0 & =h_{J_0}(0, x_0). \nonumber
\end{align}
By similar arguments to those in Section~\ref{sec:pathL}, we have that the unique solution to the above SDE is satisfied by $\mathcal{S}^{\lambda}$ constructed as follows. Let $\mathcal{Y}^{\lambda}_{i,b,r} = \{Y^{\lambda}_{i, b ,r}(t)\}_{t \geq 0}$ be given by  
\begin{align}
Y^{\lambda}_{i,b,r}(t) & := b + \int_0^t \mu_i^*\left(r+u, Y^{\lambda}_{i,b,r}(u) + F^{\lambda}_{r+u} - F^{\lambda}_r\right) \dd u. \label{eqn:Yhat}
\end{align} 
Next, define $\mathcal{S}^{\lambda} =\{S^{\lambda}_t\}_{t\ge 0}$, with $S^{\lambda}_0=x_0$, as follows  
\begin{align}
S^{\lambda}_t & := Y^{\lambda}_{J_n, S^{\lambda}_{t_n}, t_n}(t-t_n) + F^{\lambda}_t  - F^{\lambda}_{t_n}  \quad\mbox{for } t\in(t_n,t_{n+1}), n \geq 0, \\
S_{t_n}^{\lambda} &:= h_{J_n}\left(t_n, h_{J_{n-1}}^{-1}\left(t_n, S^{\lambda}_{t_n^-}\right)\right), \quad n\ge 1. 
\end{align}

Because of uniqueness, $\mathcal{S}$ is the same as $\mathcal{L}$, and similarly $\mathcal{S}^{\lambda}$ as $\mathcal{L}^{\lambda}$. To emphasize their pathwise constructions, we use the notation $\mathcal{S}$ and $\mathcal{S}^{\lambda}$ for the rest of the paper. Here, we assess the rate of strong convergence of $\mathcal{S}^{\lambda}$ to $\mathcal{S}$; in Section~\ref{sec:rateX}, we apply the inverse transformation to show strong convergence from $\mathcal{L}^{\lambda}$ to $\mathcal{L}$.

To prove that 
	\begin{align*} 
		\lim_{\lambda\rightarrow\infty}\sup_{t\in [0,T]}\left| S_t - S^{\lambda}_t\right|=0 \quad \mbox{ a.s. } \quad \mbox{ for all }  T >0, 
	\end{align*} 
	and to compute its rate of convergence, without loss of generality we assume $T=1$. 
	
	\subsection{Two additional assumptions}
	
In addition to Assumptions \ref{ass:existenceSDEs} and \ref{ass:wongzakai}, we need the following:
%\footnote{Which I actually think is necessary: after making the Poissonian jumps fit with the convergence rate, there is no more room left for ``bounding'' other types of randomness.}
\begin{Assumption}
	\label{ass:mustarunifLip}
For $i\in\mathcal{E}$, the function $\mu_i^*$ defined in (\ref{eq:mustar1}) is Lipschitz continuous on its second entry:
 \begin{align*}
  \Big|\mu_i^*(t,x_1)-\mu_i^*(t,x_2)\Big| &\le M^*_i |x_1-x_2|\quad \forall t\ge 0, x\in\mathds{R}.
  \end{align*}
\end{Assumption}
\noindent The following is an important special case for which Assumption~\ref{ass:mustarunifLip} holds, and shows that the time--inhomogeneous case is thus not a ``direct'' extension of the time--homogeneous case, {in the sense that} some extra conditions may be needed. 

\begin{Lemma}
  \label{prop:mustar5}
If $\mu_i/\sigma_i$ and $\partial_2[\sigma_{i}]$ are Lipschitz continuous on its second entry, then $\mu_i^*$ is Lipschitz continuous on its second entry.
\end{Lemma}

\begin{proof}
In the case $\mu_i/\sigma_i$ and $\partial_2[\sigma_{i}]$ are Lipschitz continuous on its second entry, then $\mu^*$ is formed by sums and second-entry compositions of functions which are Lipschitz continuous on its second entry, implying that $\mu^*$ is Lipschitz continuous on its second entry.
\end{proof}
% 
%The extra restrictions in the special case of Proposition \ref{prop:homogenous1} can be translated as having time-homogeneity in $\sigma_i$ plus strengthening the local Lipschitz condition $\mu_i/\sigma_i$ to uniform Lipschitz. 
From {here} on, we assume Assumption \ref{ass:mustarunifLip}, not necessarily in the context of Lemma~\ref{prop:mustar5}. Next, define 
	\begin{align} 
		\label{eqn:jumpscounter}
		N :=\#\{s\in (0,1]: J_{s}\neq J_{s^-}\},
	\end{align} 
the number of jumps of $\mathcal{J}$ on the interval $(0,1]$. We also assume the following.
 
\begin{Assumption}
	\label{ass:tailN}
There exists some $\gamma_0>0$ such that $\mathds{P}(N>n)= o\left(e^{- n(\log n - \gamma_0)}\right)$. 
\end{Assumption}
Assumption~\ref{ass:tailN} trivially holds if $\mathcal{J}$ is deterministic. The following lemma shows how Assumption~\ref{ass:tailN} arises naturally in common scenarios for {random jump processes}.
\begin{Lemma}
	\label{th:Poissonnlog1}
Suppose that $N$ is stochastically dominated by $N^*\sim$\emph{Poisson}$(c)$, where $ c<\infty$. Then, Assumption \ref{ass:tailN} holds with $\gamma_0=\log c + 1$. 
\end{Lemma}

\begin{proof}
By definition, stochastic domination means $ \mathds{P}(N\ge n) \le \mathds{P}(N^*\ge n)$ for $n\ge 0$. The asymptotic tail behaviour of Poisson distribution \cite[Cor. 1]{glynn1987upper} implies that $
\mathds{P}(N^*\ge n)=O\left(c^n/n!\right)$. By Stirling's formula \cite[p.~52]{feller1968introduction},
\[
	\mathds{P}(N^*\ge n)=O\left(e^{n\log c - (n+1/2)\log n + n}\right) = O\left(n^{-1/2}e^{- n(\log n-(\log c + 1))}\right), 
	\]
which completes the proof. 
\end{proof}
Lemma~\ref{th:Poissonnlog1} implies that Assumption~\ref{ass:tailN} holds whenever $\mathcal{J}$ has a bounded jump intensity. This follows by choosing a sufficiently large $c>0$ such that each holding time of $\mathcal{J}$ stochastically dominates an exponential variable with rate~$c$, implying that $N^*\sim \mbox{Poisson}(c)$ stochastically dominates $N$. Thus, Lemma~\ref{th:Poissonnlog1} allows us to work with jump processes that are not necessarily time--homogeneous Markovian, such as time--inhomogeneous Markovian or semi--Markovian processes.

%In the following we use a method which will yield the tightest bounds possible in the case $\mathcal{J}$ has Poissionian-like arrivals, that is, whenever the number of jumps can be stochastically dominated by a Poisson process.

\subsection{From $\mathcal{S}^{\lambda}$ to $\mathcal{S}$} 

We write $\Vert \mathcal{U}\Vert_{\infty}$ to denote $\sup_{t\in [0,1]}|U_t|$ for a process $\mathcal{U}=\{U_t\}_{t\ge 0}\in \mathcal{D}(0,\infty)$. 
%because {Skorokhod norm gives us the completeness property which is useless for us.}
The following lemma bounds the difference between two processes $\mathcal{Y}^{\lambda}_{i, a,r}$ and $\mathcal{Y}_{i, b,r}$ over the time interval $[0,1]$, given that the two processes have different starting points.

\begin{Lemma}
	\label{th:inbetween1}
Let $a,b\in\mathds{R}$ and $r\in [0,1]$. Then, for all $i \in \mathcal{E}$,  
\begin{align}
\Big|Y^{\lambda}_{i, a,r}(t) - Y_{i, b,r}(t)\Big| \le K_1\left(\Vert\mathcal{F}^\lambda-\mathcal{B}\Vert_\infty+|a-b|\right) \quad \forall t\in [0,1-r], \label{eq:Gronwallzeta1}
\end{align}
where $K_1 := (2\overline{M} \vee 1)\exp(\overline{M})$ and $\overline{M} := \max_{i\in\mathcal{E}} M^*_i$. 

\end{Lemma}
\begin{proof}
For $t\in [0,1-r]$,
\begin{align*}
& \Big | Y^{\lambda}_{i, a,r}(t) - Y_{i, b,r}(t) \Big|\\
&  \le \int_0^t \Big|\mu_i^*\left(r+s, Y^{\lambda}_{i,a,r}(s) + F^{\lambda}_{r+s} - F^{\lambda}_r\right)-\mu_i^*\left(r+s, Y^{\lambda}_{i,b,r}(s) + B_{r+s} - B_r\right)\Big| \dd s + |a-b|\\
& \le \overline{M}\left(\int_{0}^t \left|Y^{\lambda}_{i, a,r}(s) - Y_{i, b,r}(s)\right| + \left|\left(F^\lambda_{r+s} -F^\lambda_r \right) - \left(B_{r+s} -B_r\right)\right|\dd s\right)+ |a-b| \\
& \le  \overline{M}\int_{0}^t \left|Y^{\lambda}_{i, a,r}(s) - Y_{i, b,r}(s)\right|\dd s + 2\overline{M}\Vert\mathcal{F}^\lambda-\mathcal{B}\Vert_\infty+|a-b|.
\end{align*}
Gronwall's Inequality implies
\begin{align*}
\Big |Y^{\lambda}_{i, a,r}(t) - Y_{i, b,r}(t) \Big| \le \left(2\overline{M}\Vert\mathcal{F}^\lambda-\mathcal{B}\Vert_\infty+|a-b|\right)\exp(\overline{M}).
\end{align*}
Take $K_1 := (2\overline{M} \vee 1)\exp(\overline{M})$ and (\ref{eq:Gronwallzeta1}) follows.
\end{proof}

In general, $\mathcal{S}$ and $\mathcal{S}^{\lambda}$ do not coincide at jump epochs $\{t_k\}_{k\ge 1}$; Lemma~\ref{th:inbetween1} helps bounding the error \emph{accumulated} at each jump time up to $T=1$, as it is shown next.

\begin{theorem}\label{th:WakoLamp1} 
% 
%There exist $K_2>0$ and $K_3\ge 1$ such that 
On the event $\{N= u\}$, for $u\in\mathds{N}$, 
\begin{align}
	\label{eq:WakolbingerLamp1}
\Vert \mathcal{S}^{\lambda} - \mathcal{S} \Vert_\infty\le K_2K_3^{u}\Vert\mathcal{F}^\lambda-\mathcal{B}\Vert_\infty,
\end{align}
where $N$ is the number of jumps of $\mathcal{J}$ on the time interval $[0,1]$ given by \eqref{eqn:jumpscounter}, and $K_2 := K_3 (2+K_1)/(K_3 -1)$ and $K_3 :=K_1 V / v$. 
\end{theorem}
\begin{proof}
Let $n\in\{0,1,\dots, u\}$ and fix $t\in [t_n,t_{n+1}\wedge 1)$. Then, 
\begin{align}
	 \left|S^{\lambda}_t - S_t\right| &  =\left|Y^{\lambda}_{J_n, S^{\lambda}_{t_n},t_n}(t-t_n) + F^\lambda_t  -F^\lambda_{t_n} - \left(Y_{J_n, S_{t_n},t_n}(t-t_n) + B_t -B_{t_n} \right)\right|
\nonumber	 
	 \\
&\le \left|Y^{\lambda}_{J_n, S^{\lambda}_{t_n},t_n}(t-t_n)-Y_{J_n, S_{t_n},t_n}(t-t_n)\right| + 2\Vert \mathcal{F}^\lambda - \mathcal{B}\Vert_\infty \nonumber\\
& \le K_1\left(\Vert\mathcal{F}^\lambda - \mathcal{B}\Vert_\infty + \left|S^{\lambda}_{t_n} - S_{t_n}\right|\right) + 2\Vert \mathcal{F}^\lambda - \mathcal{B}\Vert_\infty,  \label{eqn:penguin}
\end{align} 
by Lemma~\ref{th:inbetween1}. By definition, 
\begin{align} 
	\left| S^{\lambda}_{t_n} - S_{t_n}\right| &= \left|h_{J_n}\left(t_n, h_{J_{n-1}}^{-1}(t_n, S^{\lambda}_{t_n^-})\right) - h_{J_n}\left(t_n, h_{ J_{n-1}}^{-1}(t_n, S_{t_n^-})\right) \right| \nonumber \\
	& \leq \frac{1}{v} \left| h_{J_{n-1}}^{-1} \big(t_n, S^{\lambda}_{t_n^-}\big) - h_{ J_{n-1}}^{-1}\left(t_n, S_{t_n^-}\right) \right| \nonumber \\
	& \leq \frac{V}{v} \left| S^{\lambda}_{t_n^-} - S_{t_n^-} \right|,  \label{eqn:boundStn}
\end{align}
where for the last two inequalities we used the Lipschitz continuity on the second entry of $h_i$ and $h_i^{-1}$, respectively, which is proven in Lemma~\ref{lem:hhinvLipschitz}(i). 

Let $\rho := K_1 V / v$. Substituting~\eqref{eqn:boundStn} into \eqref{eqn:penguin}, applying the same steps recursively for other jump epochs, and using $S^{\lambda}_0 - S_0 = 0 $, we have 
\begin{align*}     
\left|S^{\lambda}_t - {S}_t\right|
& \le \sum_{k=0}^{n}\rho ^k (2+K_1) \Vert\mathcal{F}^\lambda - \mathcal{B}\Vert_\infty\\
& = \frac{\rho^{u+1} -1}{\rho -1} (2+K_1)\Vert\mathcal{F}^\lambda - \mathcal{B}\Vert_\infty\\
&\le \left[\frac{\rho (2+K_1)}{\rho -1} \right]\rho^{u}\Vert\mathcal{F}^\lambda - \mathcal{B}\Vert_\infty,
\end{align*}
so that (\ref{eq:WakolbingerLamp1}) follows.
\end{proof}

Theorem \ref{th:WakoLamp1} implies that the error bound increases geometrically with respect to the number of jump occurrences in $[0,1]$. In general this would be problematic because the error keeps building up with each jump of $\mathcal{J}$---this is the main challenge when dealing with regime switching. However, as will be seen in Theorem~\ref{th:Wakolnlog1}, if the tail distribution of $N$ is \emph{light enough}, the error bound given by (\ref{eq:WakolbingerLamp1}) will still be controllable. Recall that Assumption~\ref{ass:tailN} gives the property of the tail distribution being light enough. 

Now we are ready to assess the rate of convergence  of $\mathcal{S}^{\lambda}$ to $\mathcal{S}$.
\begin{theorem}
	\label{th:Wakolnlog1}
Under Assumptions \ref{ass:existenceSDEs}--\ref{ass:tailN}, suppose that there exists $\delta:\mathds{R}_+\mapsto\mathds{R}_+$ with $\lim_{\lambda\rightarrow \infty}\delta(\lambda)=0$ such that for all $q>0$
\[
	\mathds{P}\left(\Vert \mathcal{F}^\lambda - \mathcal{B}\Vert_\infty \ge \alpha \delta(\lambda)\right)=o(\lambda^{-q}), 
\]
where $\alpha=\alpha(q)>0$ is some constant dependent on $q$ only. Then, there exists some constant $\beta=\beta(q)>0$ such that for all $\varepsilon > 0$
\begin{align} 
	\mathds{P}\left(\Vert \mathcal{S}^{\lambda} - \mathcal{S}\Vert_\infty \ge \beta\delta(\lambda)\lambda^{\varepsilon}\right)=o\left(\lambda^{-q}\right).
\end{align} 
\end{theorem}
Theorem~\ref{th:Wakolnlog1} implies that under Assumptions~\ref{ass:existenceSDEs}--\ref{ass:tailN}, as $\lambda \rightarrow \infty$ the unique solution $\mathcal{S}^\lambda$ of the regime-switching, time-inhomogeneous SDE~(\ref{eq:approx}) converges strongly to $\mathcal{S}$.

\begin{proof}
We will determine the value of $\beta=\beta(q)$ towards the end of the proof; for now let it be arbitrary but fixed. Define $f_{\gamma_0, q} :=  \exp( x (\log x-\gamma)/q)$ and let $A_{\lambda} :=f_{\gamma_0,q}^{-1}(\lambda)$, where $\gamma_0$ is from Assumption~\ref{ass:tailN}.  Let $W$ denote \emph{the Lambert $W$ function}, the inverse of the mapping $z\mapsto ze^z$. By Lemma~\ref{eq:lambertf1} in the Appendix,  
	\begin{align*} 
	A_{\lambda} =e^{\gamma_0}\exp\left\{W(q e^{-\gamma_0} \log \lambda)\right\}, 
	\end{align*} 
which implies $e^{A_{\lambda}((\log A_{\lambda}) - \gamma_0)/q}=\lambda$. By Assumption~\ref{ass:tailN}, 
\begin{align}
	\label{eq:tailofN1}
	\mathds{P}(N> A_{\lambda}) = o\left(e^{-A_{\lambda}((\log A_{\lambda}) - \gamma_0)}\right)=o(\lambda^{-q}).
\end{align}
Next, fix $\varepsilon>0$. By Theorem~\ref{th:WakoLamp1},
\begin{align}
\mathds{P}&\left(\Vert \mathcal{S}^{\lambda} - \mathcal{S}\Vert_\infty \ge \beta \delta(\lambda)\lambda^{\varepsilon}\right)\nonumber\\
& \leq \mathds{P}\left(N\le A_{\lambda},\; \Vert \mathcal{S}^{\lambda} - \mathcal{S}\Vert_\infty \ge \beta\delta(\lambda)\lambda^{\varepsilon}\right) + \mathds{P}\left(N > A_{\lambda}\right)\nonumber\\
& \le \mathds{P}\left(K_2K_3^{A_{\lambda}} \Vert \mathcal{F}^\lambda - \mathcal{B}\Vert_\infty \ge \beta\delta(\lambda)\lambda^{\varepsilon}\right) + \mathds{P}\left(N > A_{\lambda}\right)\nonumber\\
&= \mathds{P}\left(K_2 K_3^{A_{\lambda}}\lambda^{-\varepsilon} \Vert \mathcal{F}^\lambda - \mathcal{B}\Vert_\infty \ge \beta \delta(\lambda)\right) + \mathds{P}\left(N > A_{\lambda}\right).\label{eq:aux4}
\end{align}
Note that
\begin{align}
\lim_{\lambda\rightarrow\infty}K_3^{A_{\lambda}}\lambda^{-\varepsilon}&=\lim_{\lambda\rightarrow\infty}\exp\left(A_{\lambda} \log K_3 - \varepsilon \log\lambda\right). \label{eqn:koala}
\end{align}
By the asymptotic decomposition of Lambert W functions \cite[Eqn (4.18)]{corless1996lambertw} given by $W(x)= \log x - \log(\log x) + o(1)$ as $x\rightarrow\infty,$ we obtain
\begin{align} 
	\lim_{\lambda \rightarrow \infty} A_{\lambda} & = \lim_{\lambda \rightarrow \infty} e^{\gamma_0}\exp\left\{\log\left(q e^{-\gamma_0} \log \lambda\right) - \log\left(\log\left(q e^{-\gamma_0} \log \lambda\right)\right) + o(1)\right\} \nonumber \\
	 & = \lim_{\lambda \rightarrow \infty} e^{\gamma_0} \frac{q e^{-\gamma_0} \log \lambda}{\log\left(q e^{-\gamma_0} \log \lambda\right)}. \label{eqn:minilim}
\end{align}
Substituting~\eqref{eqn:minilim} into~\eqref{eqn:koala} gives 
\begin{align*}
\lim_{\lambda\rightarrow\infty} & K_3^{A_{\lambda}}\lambda^{-\varepsilon} =  \exp\left(e^{\gamma_0} \log K_3 \lim_{\lambda \rightarrow \infty} \left\{ \frac{q \log K_3}{\log\left(q e^{-\gamma_0} \log \lambda\right)} - \varepsilon\right\} \log\lambda\right) = 0. 
\end{align*}
Let $\lambda_0=\lambda_0(\varepsilon)>0$ be such that $K_3^{A_{\lambda}}\lambda^{-\varepsilon}\le 1$ for all $\lambda\ge\lambda_0$. Take $\beta:=K_2K_1\alpha$, then for all $\lambda\ge \lambda_0$,
\begin{align*}
	\mathds{P}\left(K_2 K_3^{A_{\lambda}}\lambda^{-\varepsilon}\Vert \mathcal{F}^\lambda - \mathcal{B}\Vert_\infty \ge \beta \delta(\lambda)\right)&\le \mathds{P}\left(\Vert \mathcal{F}^\lambda - \mathcal{B}\Vert_\infty \ge \alpha\delta(\lambda)\right)\\
  & =o(\lambda^{-q}) \quad\mbox{as}\quad\lambda\rightarrow\infty,
\end{align*}
and so by (\ref{eq:tailofN1}) and (\ref{eq:aux4}) we complete the proof. 
\end{proof}
% 
%Theorem \ref{th:Wakolnlog1} and Theorem \ref{th:Poissonnlog1} imply the following.
% \begin{Corollary}\label{cor:Wakolinhom1}Let $q>0$. If $\mathcal{J}$ is an inhomogeneous Markov jump process with finite state space and bounded jump intensity, then there exists $\beta(q)$ such that
% \[\mathds{P}\left(\Vert \mathcal{S}_{\mathcal{F}^\lambda,\mathcal{J}} - \mathcal{S}_{\mathcal{B},\mathcal{J}}\Vert_\infty \ge \beta(q) \delta(\lambda)\lambda^{\varepsilon}\right)=o(\lambda^{-q})\]
% for all $\varepsilon>0$.
% \end{Corollary}

% 

% 
\section{Rate of strong convergence from $\mathcal{X}^{\lambda}$ to $\mathcal{X}$}
	\label{sec:rateX}
Now that we have obtained a rate of convergence of $\Vert \mathcal{S}^\lambda-\mathcal{S}\Vert_\infty$, we return to the original setting: approximating the solution $\mathcal{X}$ to the regime-switching time-inhomogeneous stochastic differential equation~\eqref{eq:RSSDEX1}.

In the following we assess the distance between $\mathcal{X}$ and $\mathcal{X}^\lambda$ of (\ref{eq:RSSDEXlambda3}) in terms of the distance between their respective transformations, $\mathcal{S}$ and $\mathcal{S}^{\lambda}$.
\begin{Proposition}
	\label{prop:aux7}
\begin{align}
	\label{eq:auxVLip4}
	\Vert \mathcal{X}^\lambda - \mathcal{X}\Vert_\infty\le V\Vert \mathcal{S}^{\lambda} - \mathcal{S}\Vert_\infty, 
	\end{align}
where $V$ is the uniform upper bound for $\sigma_i$, $i \in \mathcal{E}$. 
\end{Proposition}

\begin{proof} 
As $X_t= h_{J_t}^{-1}(t, S_t)$ and $X_t^\lambda= h_{J_t}^{-1}(t, S^{\lambda}_t)$, Inequality~(\ref{eq:auxVLip4}) follows from the Lipschitz continuity on the second entry of $h_i$, as shown in Lemma~\ref{lem:hhinvLipschitz}(i).
\end{proof}
 
Theorem~\ref{th:mainmainresult} then follows directly from Theorems \ref{th:Wakolnlog1} and Proposition \ref{prop:aux7}.

\begin{Remark}
\emph{Even though Theorem \ref{th:Wakolnlog1} and Proposition \ref{prop:aux7} concern the case $T=1$, their generalization to arbitrary $T >0$ is straightforward and thus Theorem \ref{th:mainmainresult} indeed follows.}
\end{Remark}

\section{Special case: Markov-modulated Brownian motion process}
	\label{sec:applications}
	
Here, we focus on the special case where the finite--variation process $\mathcal{F}^\lambda=\{F^\lambda_t\}_{t\ge 0}$ is a \emph{uniform transport process}~\cite{griego1971almost} (also known as the \emph{telegraph process}~\cite{k74} or \emph{flip-flop process}~\cite{latouche2015morphing}): 
\[
	F^\lambda_t := \sqrt{\lambda}\int_0^t(-1)^{M_s^\lambda}\dd s,
\]
where $\{M_t^\lambda\}_{t\ge 0}$ is a Poisson process of intensity $\lambda>0$. Through two different constructions, Gorostiza and Griego~\cite{gorostiza1980rate} and Nguyen and Peralta~\cite{nguyen2019strong} showed that there exists a probability space on which $\mathcal{F}^\lambda$ converges strongly to $\mathcal{B}$, at rates $\delta_1(\lambda)=\lambda^{-1/4}(\log\lambda)^{5/2}$ and $\delta_2(\lambda)=\lambda^{-1/4}\log\lambda$, respectively. 

Now, for $i\in\mathcal{E}$, let $\mu_i$ and $\sigma_i$ be fixed constants. Then, $\mathcal{X}$ and $\mathcal{X}^\lambda$ respectively take the form
\begin{align*}
X_t & = x_0 + \int_0^t \mu_{J_s}\dd s + \int_0^t \sigma_{J_s}\dd B_s,\quad\mbox{and}\\
X_t^\lambda & = x_0 + \int_0^t \mu_{J_s}\dd s + \int_0^t \sigma_{J_s}\dd F_s^\lambda.
\end{align*}
When $\mathcal{J}$ is a time-homogeneous Markovian jump process independent of $\mathcal{B}$, the process $\mathcal{X}$ is called a \emph{Markov-modulated Brownian motion}. In this context, Latouche and Nguyen~\cite{latouche2015morphing} proved that $\mathcal{X}^\lambda$ converges weakly to $\mathcal{X}$. Theorem \ref{th:mainmainresult}, together with Theorem~3.1 in~\cite{np20}, implies that such convergence also holds in a strong sense. 

Significantly, Theorem \ref{th:mainmainresult} is applicable in much more general settings, not only in terms of the drift and diffusion coefficients but also in terms of the jump process~$\mathcal{J}$; the latter can be, for example, \emph{time--inhomogeneous} or \emph{semi--Markov}.

\section*{Acknowledgement} 
The authors acknowledge the funding of the Australian Research Council Discovery Project DP180103106. 

\section{Appendix}
	\label{sec:appendix}
	
	\subsection{Properties of the Lamperti transformation} 
The following are some standalone results needed throughout the main text.
\begin{Lemma}
	\label{lem:hhinvLipschitz}
Suppose Assumptions \ref{ass:existenceSDEs} and \ref{ass:wongzakai} hold. Then, for $i\in\mathcal{E}$,
\begin{enumerate}
  \item[\emph{(i)}] $h_{i}$, $h_{i}^{-1}$ and $\partial_1[h_i]$ are Lipschitz continuous on their second entry,
  \item[\emph{(ii)}] $h_i$ 
is locally Lipschitz continuous and of linear growth,
\item[\emph{(iii)}] $h^{-1}$ 
is locally Lipschitz continuous and of linear growth, where $h^{-1}(t, y) :=  x$ such that $h(t, x) = y$. 
\item[\emph{(iv)}] $\partial_1[h_i]$ is locally Lipschitz continuous and of linear growth,
\item[\emph{(v)}] $\mu_i/\sigma_i$ is locally Lipschitz continuous and of linear growth.
\end{enumerate} 
\end{Lemma}

\begin{proof}
\textbf{(i)} By Assumption~2(i), there exist $V<\infty, v>0$ such that $v<\sigma_i<V$.
Then,
  \begin{align*}
  	\Big| h_i(t, x_1)-h_i(t, x_2) \Big|=\left|\int_{x_1}^{x_2}\frac{1}{\sigma_{i}(t,y)}\dd y\right|\le \frac{1}{v} |x_1 - x_2|,\quad  t\ge 0,\, x_1,x_2\in\mathds{R},
  \end{align*}
which implies $h_{i}$ is Lipschitz continuous on its second entry.

Next, 
  \begin{align}
  	\label{eq:Lipaux2}
	\Big|h_{i}(t, x_1) - h_i(t, x_2)\Big| =\left|\int_{x_1}^{x_2}\frac{1}{\sigma_{i}(t,y)}\dd y\right|\ge \frac{1}{V} |x_1 - x_2|,\quad t\ge 0,\, x_1,x_2\in\mathds{R}. 
\end{align}
Substituting $x_1:= h_{i}^{-1}(t, y_1)$ and $x_2 :=h_{i}^{-1}(t, y_2)$ into (\ref{eq:Lipaux2}) implies that 
\begin{align*}
	\Big| h^{-1}_i(t, y_1) - h^{-1}_i(t, y_2)\Big| \leq V |y_1 - y_2|, \quad t\ge 0,\, y_1,\, y_2 \in \mathds{R},
\end{align*} 
  meaning $h_{i}^{-1}$ is Lipschitz continuous on its second entry.

By Assumption \ref{ass:wongzakai}(ii),
  \begin{align}
  \Big| \partial_1\left[h_i(t,x_1)\right] - \partial_1\left[h_i(t,x_2)\right] \Big| & = \left|\int_{x_1}^{x_2}\frac{\partial_1\left[\sigma_i(t,y)\right]}{\sigma_i^2(t,y)}\dd y\right|\nonumber\\
  & \le K |x_1-x_2|, \quad t\ge 0,\, x_1,x_2\in\mathds{R},\label{eq:Lippartialh}
 \end{align}
 implying that $\partial_1\left[h_i\right]$ is Lipschitz continuous on its second entry.

For Parts (ii)--(v), without loss of generality we focus on a compact set $\mathcal{K} = [t_1, t_2] \times [y_1, y_2]$, where $0 \leq t_1 \leq t_2$, $y_1 \leq x_0 \leq y_2$, and let $(s_1, x_1), (s_2, x_2) \in \mathcal{K}$.
  
 \textbf{(ii)} As $h_i$ is Lipschitz continuous on its second entry, 
  \begin{align*}
 \Big |h_i(s_1,x_1) - h_i(s_2,x_2)\Big| & \le\Big |h_i(s_1,x_1) - h_i(s_1,x_2)\Big| +  \Big|h_i(s_1,x_2) - h_i(s_2,x_2)\Big|\\
  & \le \frac{1}{v}|x_1-x_2| + \Big| h_i(s_1,x_2) - h_i(s_2,x_2)\Big|.
  \end{align*}
  Since $|\partial_1[\sigma_i] |/\sigma_i^2\le K $ by Assumption \ref{ass:wongzakai}(ii), 
\begin{align}\label{eq:lineargrpartialh1}
\Big| \partial_1[h_i(t,x)]\Big| = \left|\int_{x_0}^x \frac{\partial_1[\sigma_i(t,y)]}{\sigma_i^2(t,y)}\dd y\right|\le K|x-x_0| .
\end{align}
The mean value theorem gives $
|h_i(s_1,x_2) - h_i(s_2,x_2)|\le K |x_2-x_0| |s_1-s_2|, 
$
which leads to 
\begin{align}
\Big|h_i(s_1,x_1) - h_i(s_2,x_2)\Big| & \le \frac{1}{v}|x_1-x_2| + K |x_2-x_0| |s_1-s_2| \nonumber \\
& \le \frac{1}{v}|x_1-x_2| + K |y_1-y_2| |s_1-s_2| 	\label{eq:apphi1}
\end{align}
and thus $h_i$ is locally Lipschitz. The Lipschitz continuity on the second entry of $h_{i}$ implies that
\[
	\Big |h_i(t,x)\Big| \le \Big|h_i(t,x_0)\Big| + \frac{1}{v}|x - x_0| =  \frac{1}{v}|x - x_0|,
\]
so that $h_i$ has linear growth.

\textbf{(iii)} We have
\begin{align}
  \Big| h_i^{-1}(s_1,x_1) - h_i^{-1}(s_2,x_2)\Big| & \le \Big|h_{i}^{-1}(s_1, x_1) - h_{i}^{-1}(s_1, x_2)\Big| +  \Big|h_{i}^{-1}(s_1, x_2) - h_{i}^{-1}(s_2, x_2)\Big|  \nonumber \\
  & \le V\left|x_1-x_2\right| +  \Big|h_{i}^{-1}(s_1, x_2) - h_{i}^{-1}(s_2, x_2)\Big|. \label{eqn:hinv}
  \end{align}
  Note that $h_i^{-1}(t, 0) =x_0$; furthermore, 
  \begin{align*}
  \partial_2 \left[h_i^{-1}(t,x)\right] &  = \frac{1}{\partial_2 \left[h_{i}\left(t, h_{i}^{-1}(t, x)\right)\right]} = \sigma_i \left(t,h_{i}^{-1}(t, x)\right),
  \end{align*}
Thus, 
\begin{align}	
  	h_i^{-1}(t, x) = x_0+\int_0^x\sigma_i\left(t,h_{i}^{-1}(t, y)\right)\dd y.
\end{align}
  Without loss of generality, suppose that $x_2\ge 0$. Then
  \begin{align*}
  \Big|h_{i}^{-1}(s_1, x_2) - h_{i}^{-1}(s_2, x_2)\Big| &\le \int_0^{x_2} \Big|\sigma_i\left(s_1,h_{i}^{-1}(s_1, y)\right) - \sigma_i\left(s_2,h_{i}^{-1}(s_2, y)\right)\Big|\dd y\\
  & \le \int_0^{x_2} G_\mathcal{K} \left(\left|s_1-s_2\right| + \Big| h_{i}^{-1}(s_1, y) - h_{i}^{-1}(s_2, y)\Big|\right)\dd y\\
  & \le G_\mathcal{K} x_2 |s_1-s_2|  + G_\mathcal{K} \int_0^{x_2} \Big|h_{i}^{-1}(s_1, y) - h_{i}^{-1}(s_2, y)\Big|\dd y,
  \end{align*}
  where, by Assumption~\ref{ass:existenceSDEs}(ii), $G_\mathcal{K}$ is the local Lipschitz constant of $\sigma_i$ on $\mathcal{K}$.
  
  Gronwall's Inequality implies that
\begin{align}
	\label{eqn:gronineq}
  	\Big| h_{i}^{-1}(s_1, x_2) - h_{i}^{-1}(s_2, x_2) \Big| \le G_\mathcal{K} x_2|s_1-s_2| \exp(G_\mathcal{K} x_2).
\end{align}
Substituting \eqref{eqn:gronineq} into \eqref{eqn:hinv} gives  
  \begin{align}
  \Big| h_i^{-1}(s_1,x_1) - h_i^{-1}(s_2,x_2) \Big| & \le V|x_1-x_2| + G_\mathcal{K} x_2|s_1-s_2| \exp(G_\mathcal{K} x_2) \nonumber \\ 
  & \le V|x_1-x_2| + G_\mathcal{K} y_2 |s_1-s_2| \exp(G_\mathcal{K} y_2)
  	\label{eq:apphiinv1}
  \end{align}
  and so $h_i^{-1}$ is locally Lipschitz continuous.
  
  The linear growth of $h_i^{-1}$ is implied by the inequality $\Big |h_{i}^{-1}(t,x)\Big | \le |x_0| + V |x|.$

\textbf{(iv)} Note that, by (\ref{eq:Lippartialh}),
\begin{align*}
\Big|\partial_1 \left[h_i(s_1,x_1)\right] - \partial_1 \left[h_i(s_2,x_2)\right] \Big|& \le \Big|\partial_1\left[h_i(s_1,x_1)\right] - \partial_1\left[h_i(s_1,x_2)\right] \Big|  +  \Big|\partial_1\left[h_i(s_1,x_2)\right] - \partial_1\left[h_i(s_2,x_2)\right] \Big|\\
  &\le K|x_1-x_2|  +  \Big|\partial_1\left[h_i(s_1,x_2)\right] - \partial_1\left[h_i(s_2,x_2)\right] \Big|.
  %\label{eqn:diffpart}
  \end{align*}
As $v\le \sigma_i\le V$ and $\sigma_i$ is locally Lipschitz, $1/\sigma_i$ is also locally Lipschitz. Since the product of locally Lipschitz is locally Lipschitz, by Assumption~\ref{ass:existenceSDEs}(ii, iii) the function $\partial_1[\sigma_i]/\sigma_i^2$, too, is locally Lipschitz. This implies the existence of $Q_\mathcal{K}>0$ such that
\begin{align*}
  \left|\frac{\partial_1 \left[\sigma_i(s_1,x_1)\right] }{\sigma_i^2(s_1,x_1) }- \frac{\partial_1\left[\sigma_i (s_1,x_1)\right] }{\sigma_i^2(s_1,x_1) } \right|\le Q_\mathcal{K} \left(|s_1 - s_2| + |x_1 - x_2|\right)
\end{align*}
    Thus,
  \begin{align}
  \Big|\partial_1\left[h_i(s_1,x_2)\right] - \partial_1\left[h_i(s_2,x_2)\right] \Big| &= \left|\int_{x_0}^{x_2}\left\{ \frac{\partial_1[\sigma_i(s_1,y)]}{\sigma_i^2(s_1,y)} - \frac{\partial_1[\sigma_i(s_2,y)]}{\sigma_i^2(s_2,y)}\right \} \dd y\right| \nonumber \\
  &  \le \left|\int_{x_0}^{x_2}Q_\mathcal{K}|s_1-s_2|\dd y\right| \nonumber \\
  & = Q_\mathcal{K}|x_0-x_2||s_1-s_2|  \nonumber \\
  &  \le Q_\mathcal{K}(y_2 - y_1)|s_1-s_2|. \nonumber
  \end{align}
  % 
%Substituting \eqref{eqn:firstterm} and \eqref{eqn:secondterm} gives 
Thus,
  \begin{align*}
  	%\label{eq:auxdd1h1}
	% 
 \Big| \partial_1\left[ h_i(s_1,x_1)\right] - \partial_1\left[h_i(s_2,x_2)\right] \Big| &\le K |x_1-x_2| + Q_\mathcal{K} (y_2 - y_1)|s_1-s_2|, 
  \end{align*}
which means $\partial_1h_i$ is locally Lipschitz continuous. The linear growth condition follows from (\ref{eq:lineargrpartialh1}).
  
 \textbf{(v)} As $\mu_i/\sigma_i$ is the product of two locally Lipschitz continuous functions, it is locally Lipschitz. Its linear growth follows from the linear growth of $\mu_i$ and the boundedness of $\sigma_i$.
\end{proof}

% \begin{Corollary}
% 	\label{cor:apphomogeneous1}
% % 
% Suppose Assumptions \ref{ass:existenceSDEs} and \ref{ass:wongzakai} hold, and that $\sigma_i(s_1,x)=\sigma_i(s_2,x)$ for all $s_1,s_2\ge 0$, $x\in\mathds{R}$ and $i\in\mathcal{E}$. Then, the functions $h_{i}(t, \cdot)$, $h_{i}^{-1}(t, \cdot)$ and $\partial_1[h_{i}(t, x)]$ are Lipschitz continuous.
% \end{Corollary}
% % 
% \begin{proof}
% Simply take $s_1=s_2$ in (\ref{eq:apphi1}), (\ref{eq:apphiinv1}) and (\ref{eq:auxdd1h1}).
% \end{proof}

\subsection{Inverse of function $f_{\gamma, q}$} 

\begin{Lemma}
\label{lem:Lambert1}

For $\gamma, q > 0$, let $f_{\gamma,q}:[e^\gamma,\infty)\mapsto [1,\infty)$ be defined by 
\begin{align}
	\label{eq:lambertf1}
f_{\gamma,q}(x):=\exp( x (\log x-\gamma)/q).
\end{align}
Then $f_{\gamma,q}$ is bijective with inverse given by
\begin{align}
	\label{eq:inversef1}
	f_{\gamma,q}^{-1}(y)= e^{\gamma}\exp\left(W(q e^{-\gamma} \log y)\right),
\end{align}
where $W$ denotes \emph{the Lambert $W$ function}, the inverse of the mapping $z\mapsto ze^z$.
\end{Lemma}
\begin{proof} The bijectivity of $f_{\gamma,q}$ follows from the fact that it is strictly increasing, with $\lim_{x\rightarrow\infty}f_{\gamma,q}(x)=\infty$. Let $y$ be in the image of $f_{\gamma,q}$, so that it is of the form
\[
	y=\exp( x (\log x-\gamma)/q)\quad\mbox{for some } x\ge e^\gamma.
\]
Taking the log of both sides and multiplying $qe^{-\gamma}$ across gives
\begin{align} 
	\label{eqn:painandtorture}
	qe^{-\gamma} \log y = e^{-\gamma} x(\log x - \gamma). 
\end{align} 
Note that $e^{-\gamma} x(\log x - \gamma) = e^{-\gamma} x(\log x + \log(e^{-\gamma})) =  \log (xe^{-\gamma})e^{\log (xe^{-\gamma})}$. Let $z:=\log(xe^{-\gamma})$, we have that \eqref{eqn:painandtorture} is equivalent to $qe^{-\gamma}\log y= z e^z.$ 
Thus,
\begin{align}
	\label{eq:aux07}
	\log(xe^{-\gamma}) = W\left(qe^{-\gamma}\log y\right).
\end{align}
Solving (\ref{eq:aux07}) for $x$ gives (\ref{eq:inversef1}).
\end{proof}

\bibliographystyle{abbrv}
\bibliography{WZpaper}

\end{document}